\documentclass[reqno,12pt]{amsart}
\usepackage{verbatim}
\usepackage{amssymb}
\usepackage{enumerate}
\usepackage[active]{srcltx}
\usepackage{t1enc}
\usepackage[utf8x]{inputenc}

\numberwithin{equation}{section}

\newtheorem{theorem}{Theorem}[section]

\newtheorem{lemma}[theorem]{Lemma}
\newtheorem{corollary}[theorem]{Corollary}

\newtheorem{claim}{Claim}[theorem]

\theoremstyle{definition}

\newtheorem{definition}[theorem]{Definition}

\theoremstyle{remark}

\def\myheads#1;#2;{
\pagestyle{myheadings}
\markboth{{\sc\hfill #1\hfill\protect\makebox[0cm][r]{\rm\today}}}
{{\sc\protect\makebox[0cm][l]{\rm\today}\hfill #2\hfill}}
}

\newif\ifdeveloping

\ifdeveloping
\usepackage[notref,notcite]{showkeys}
\fi

\newif\ifcommented

\newcommand{\comm}[1]{}

\ifcommented
\renewcommand{\comm}[1]{
\fbox{\fbox{\begin{minipage}{300pt}#1\end{minipage}}
}}

\fi

\newcommand{\mc}[1]{\mathcal{#1}}

\newcommand{\setm}{\setminus}
\newcommand{\empt}{\emptyset}
\newcommand{\subs}{\subset}

\def\<{\left\langle}
\def\>{\right\rangle}
\def\br#1;#2;{\bigl[ {#1} \bigr]^ {#2} }
\newcommand{\cf}{\operatorname{cf}}

\newcommand{\intt}{\operatorname{Int}}

\newcommand{\ext}{\operatorname{e}}

\newcommand{\accu}[1]{{#1}^\circ}

\newcommand{\appr}[1]{$#1$-approachable}

\newcommand{\oo}{{\omega_1}}

\newcommand{\leka}[1]{{\overline{#1}}^{{<\kappa}}}

\newcommand{\clas}[1]{{\mc C}_{#1}}
\newcommand{\rcp}[1]{\operatorname{RC}^+(#1)}

\newcommand{\kpl}{{{\kappa}}^+}

\author[I. Juhász]{István Juhász}
\address
      { Alfréd Rényi Institute of Mathematics, Hungarian Academy of Sciences
}
\email{juhasz@renyi.hu}

\author[L. Soukup]{Lajos Soukup}
\address
      { Alfréd Rényi Institute of Mathematics, Hungarian Academy of Sciences}
\email{soukup@renyi.hu}
\urladdr{http://www.renyi.hu/$\tilde{}$soukup}

\author[Z. Szentmiklóssy]{
Zoltán Szentmiklóssy}
\address{E\"otv\"os Loránd University of Budapest}
\email{zoli@renyi.hu}

\title{Regular spaces of small extent are $\omega$-resolvable}
\thanks{The research on and preparation of this paper was
supported by  OTKA grant no. K 83726.}

\date{\today}

\begin{document}

\maketitle

\begin{abstract}
We improve some results of Pavlov and of Filatova, respectively, concerning a problem of Malychin
by showing that every regular space $X$ that satisfies $\Delta(X)>\ext(X)$
is ${\omega}$-resolvable. Here $\Delta(X)$, the dispersion character of $X$, is the
smallest size of a non-empty open set in $X$ and $\ext(X)$, the extent of $X$,
is the supremum of the sizes of all closed-and-discrete subsets of $X$.
In particular, regular Lindelöf spaces of uncountable dispersion character
are ${\omega}$-resolvable.

We also prove that
any regular Lindelöf space $X$ with $|X|=\Delta(X)=\omega_1$ is even $\oo$-resolvable.
The question if regular Lindelöf spaces of uncountable dispersion character are maximally
resolvable remains wide open.

\end{abstract}

\section{Introduction}

We start by recalling a few basic definitions and facts concerning
resolvability.
A topological space $X$ is said to be $\lambda$-resolvable ($\lambda$
a cardinal) if $X$ contains $\lambda$ many mutually disjoint dense
subsets. A natural upper bound on the resolvability of $X$ is
$$\Delta(X)=\min\{|G| : G  \text{ is non-empty open in }X\}\,,$$
called the dispersion character of $X$.
So, $X$ is said to be {\it maximally resolvable} if it is
$\Delta(X)$-resolvable. The expectation is that ``nice" spaces
should be maximally resolvable, an expectation verified e.g. by
the well-known facts that compact Hausdorff, or metric, or linearly ordered spaces
are all maximally resolvable.

It is also well-known, however,  that
there is a countable regular (hence ``nice") space with no
isolated points that is {\em irresolvable}, i.e. not even 2-resolvable.
Since countable spaces are (hereditarily) Lindelöf, this prompted
Malychin to ask the following natural question in \cite{M}: Is every
regular Lindelöf space of uncountable dispersion character (at least 2-)resolvable?
He also noted that the answer to this question is negative if regular is
weakened to Hausdorff.

Pavlov in \cite{P} proved the following very deep result that gave a
partial affirmative answer to Malychin's question: If $X$ is any regular space
satisfying $\Delta(X)>\ext(X)^+$ then $X$ is ${\omega}$-resolvable. (In fact,
he only needed the following assumption on $X$ that
we call $\pi$-regularity and is clearly weaker than regularity:
The regular closed sets in $X$ form a $\pi$-network, i.e. every non-empty open
set includes a non-empty regular closed set.)

We recall that the {\em extent} $\ext(X)$ of $X$ is the supremum of the sizes of
all closed-and-discrete subsets of $X$. Since Lindelöf spaces have countable extent,
it followed that regular Lindelöf spaces of dispersion character $> \omega_1$
are ${\omega}$-resolvable. Thus only the case $\Delta(X)=\omega_1$ of Malychin's
problem remained and that was settled by Filatova in \cite{F}: Any
regular Lindelöf space $X$ with $\Delta(X)=\omega_1$ is 2-resolvable. However,
her method of proof did not seem to give even 3-resolvable, not to mention $\omega$--resolvable
as in Pavlov's result.

Our main result in this paper, theorem \ref{tm:gen_e_res}, improves Pavlov's above result by
showing that the assumption
$\Delta(X)>\ext(X)^+$ in it can be relaxed to $\Delta(X)>\ext(X)$. This, of course, immediately
implies that Filatova's 2-resolvable can also be improved to $\omega$-resolvable.
We also think that the proof of our strengthening of Pavlov's result is significantly simpler
than Pavlov's original proof, especially in the case when $\Delta(X)$ is singular.

We do not know, however, if a regular space $X$ satisfying $\Delta(X)>\ext(X)$ is always maximally
resolvable, or even if regular Lindelöf spaces of uncountable dispersion character are maximally
resolvable. This problem should be confronted with our result from \cite{JSSz_res2} stating that
any topological space $X$ satisfying $\Delta(X)> s(X)$ is maximally resolvable.
Here $s(X)$, the spread of $X$,  is the supremum of the sizes of
all (relatively) discrete subsets of $X$.

Theorem \ref{tm:steppingup} in this paper implies that, for any infinite cardinal $\kappa$,
if all regular Lindelöf spaces of cardinality and
dispersion character $\kappa^+$ are $\kappa$-resolvable then
all such spaces are actually $\kappa^+$-resolvable
as well. This then, together with  theorem \ref{tm:gen_e_res},
implies that any regular Lindelöf space $X$ with $|X|=\Delta(X)=\omega_1$ is even $\oo$-resolvable,
i.e. maximally resolvable. Considering that after Pavlov and before Filatova this was the unsolved ``hard case"
of Malychin's problem, for which even 2-resolvability was unknown, it seems to be not unreasonable to
raise the question if regular Lindelöf spaces of uncountable dispersion character are maximally
resolvable.

\section{Preliminary results}

In this section we have collected some known and some new results that will play an
essential role in the proof of our main results, theorems \ref{tm:gen_e_res} and
\ref{tm:steppingup}.

First we fix a couple of important pieces of notation: For a topological space $X$,
we denote by $\tau^+(X)$ the collection of all non-empty open sets in $X$ and by
$\rcp{X}$ the family of all non-empty regular closed subsets of $X$. As we have noted
above, if $X$ is regular then $\rcp{X}$ is a $\pi$-network in $X$.

We shall make very frequent use of the following simple but basic result.

\begin{theorem}[Elkin, \cite{El}]\label{tm:el}
If $X$ is a topological space, ${\kappa}$ is any cardinal,
and the family
\begin{equation}\notag
\mathfrak{R}_\kappa(X) = \{\text{$Z\subs X$ : Z is ${\kappa}$-resolvable}\}
\end{equation}
is a $\pi$-network in $X$ then $X \in \mathfrak{R}_\kappa(X)$, i.e. $X$ is ${\kappa}$-resolvable.
\end{theorem}

Since for every $G \in \tau^+(X)$ there is $H \in \tau^+(X)$ such that $H \subs G$
and $|H| = \Delta(H)$, moreover then $R \in \rcp X$ and $R \subs H$ imply
$|R| = \Delta(R)\, \big(=|H|\big)$, we obtain the following simple but useful corollary.

\begin{corollary}\label{cl:rc}
Let $\mathcal{C}$ be a regular closed hereditary class of regular spaces,
i.e. such that $X \in \mc C$ implies $\rcp X \subs \mc C$. If every
space $X \in \mathcal{C}$ with $|X| = \Delta(X)$ has a $\kappa$-resolvable subspace then every
member of $\mathcal{C}$ is $\kappa$-resolvable.
\end{corollary}

In proving that certain spaces are  $\omega$-resolvable, like in the proof of
theorem \ref{tm:gen_e_res}, the following result comes naturally handy.

\begin{theorem}[Illanes, \cite{I}]\label{tm:illanes}
If a topological space $X$ is $k$-resolvable for each $k<{\omega}$
then $X$ is ${\omega}$-resolvable.
\end{theorem}

Now we turn to formulating and proving some new results that will be needed later,
in the proofs of our main theorems \ref{tm:gen_e_res} and \ref{tm:steppingup}.
They may turn out to be of independent interest. 

First we fix some, rather standard, notation: If $A$
is any subset of a topological space then $A'$ denotes the derived set of $A$, that is
the set of all accumulation points of $A$, while we use $\accu A$ to denote the set of
all {\em complete accumulation} points of $A$.

The following rather technical result is new, although it owes its basic idea to Filatova's work in \cite{F}.

\begin{lemma}\label{lm:2res}
Let  $X$ be a regular space,
$\kappa$ be a regular cardinal, and consider
the family
\begin{equation}
\mc D=\big\{D\in \br X;{\kappa};: D' = \accu D ,\, D\,\cap\, \accu D=\empt, \text{ and }\,
\forall\, E\in \br D;{\kappa};\ (\accu E\ne \empt) \big\}.
\end{equation}
If $X$ has a dense subset $Y$ with $|Y| \le \kappa$ such that for each point $y \in Y$
there is a set $D \in \mathcal{D}$ with $y \in \accu D$ then $X$ is 2-resolvable.
\end{lemma}

\newcommand{\dkor}{D^*}

\begin{proof}[Proof of the lemma]
First, let us fix a $\kappa$-type enumeration
of $Y$ (with repetitions permitted): $Y = \{y_{\alpha}:{\alpha}<{\kappa}\}$.
We shall then, by induction on ${\alpha} < {\kappa}$, define $D_{\alpha}\in \mc D$
and $i_{\alpha}\in 2$ in
such a way that, putting for ${\alpha}\le {\kappa}$ and $i<2$
\begin{equation}\notag
{E}_{{\alpha},i}=\bigcup\{D_{\beta}: {\beta}<{\alpha} \mbox{ and } i_{\beta}=i\}
\cup \bigcup\{\accu D_{\beta}: {\beta}<{\alpha}  \mbox{ and }
i_{\beta} = 1-i\},
\end{equation}

\medskip

\noindent for any $\alpha \le \kappa$ we have both
\begin{equation}\label{eq:disj_d}
{E}_{\alpha,0} \cap  {E}_{\alpha,1}=\empt
\end{equation}
and
\begin{equation}\label{eq:cover_d}
\{y_{\beta}:{\beta}<{\alpha}\}\subs \overline{{E}_{\alpha,0}}
\cap  \overline{{E}_{\alpha,1}}.
\end{equation}

\medskip

To start with, we pick $D_0 \in \mathcal{D}$ with $y_0 \in \accu D_0$ and put $i_0 = 0$.
Then (\ref{eq:disj_d}) and (\ref{eq:cover_d}) are trivially satisfied.
Now, assume that $0 < \alpha < \kappa$, moreover
$\{D_{\beta}:{\beta}<{\alpha}\}$ and $\{i_{\beta}:{\beta}<{\alpha}\}$ have been defined
and satisfy the inductive hypotheses (\ref{eq:disj_d}) and (\ref{eq:cover_d}).

We now distinguish three cases. First, if
\begin{equation}\notag
y_{\alpha} \in  \overline{ {E}_{{\alpha},0}} \cap  \overline{{E}_{{\alpha},1}}
\end{equation}
then we may simply let $D_{\alpha}=D_0$ and $i_{\alpha}=i_0  = 0$.
Clearly, then (\ref{eq:disj_d}) and (\ref{eq:cover_d}) will remain valid for $\alpha+1$.

Next, if
\begin{equation}\notag
y_{\alpha}\notin  \overline{{E}_{{\alpha},0}} \cup  \overline{{E}_{{\alpha},1}}
\end{equation}
then, using the regularity of $X$,
we can pick an open neighbourhood $U$ of $y_{\alpha}$ for which $\overline{U} \cap\,
\overline{\bigcup_{\beta<\alpha} D_\beta} =\empt$. Now choose $D \in \mc D$
with $y_{\alpha}\in \accu D$ and set $D_\alpha = U \cap D$, moreover let $i_{\alpha}=0$.
Again, it is easy to check that with these choices (\ref{eq:disj_d}) and (\ref{eq:cover_d}) remain valid.

If none of the above two alternatives hold then
$y_\alpha \in \overline{{E}_{{\alpha},0}} \vartriangle  \overline{{E}_{{\alpha},1}}$, i.e.
there is $j\in 2$ such that $y_\alpha \in \overline{{E}_{{\alpha},j}} \setm  \overline{{E}_{{\alpha},1-j}}$.
Suppose e.g. that $j=0$, the case $j=1$ can be handled symmetrically.

We may then choose an open neighbourhood $U$ of $y_{\alpha}$ such that $\overline{U} \cap \overline{E_{\alpha,1}} = \empt$
and a set $D \in \mc D$ with $y_{\alpha}\in \accu D$. For every $\beta  < \alpha$ we have $|U \cap D_\beta| < \kappa$ :
Indeed, if $i_\beta = 0$ then $|U \cap D_\beta| = \kappa$ would imply $$\empt \ne  (U \cap D_\beta)^\circ
\subs \overline{U} \cap \accu D_\beta \subs \overline{U} \cap E_{\alpha,1},$$ a contradiction.
And if $i_\beta = 1$ then we even have $U \cap D_\beta = \empt$.

Consequently, as $\kappa$ is regular, we have $|\bigcup \{D \cap D_\beta : \beta < \alpha \}| < \kappa$,
hence $D_\alpha = U \cap D\, \setm\, \bigcup \{D_\beta : \beta < \alpha \} \in \mc D$ and $y_ \alpha \in \accu D_\alpha$.
Let us now put $i_\alpha = 1$. Then $y_\alpha \in {D_\alpha}^\circ \subs \overline{D_\alpha} \subs \overline{E_{\alpha+1,1}}$ implies
$y_\alpha \in \overline{E_{\alpha+1,0}} \cap \overline{E_{\alpha+1,1}}$, hence
(\ref{eq:cover_d}) remains valid for $\alpha+1$.

To show the same for (\ref{eq:disj_d}),
note first that ${D_\beta}^\circ \subs \overline{E_{\alpha,1}}$,
hence $\overline{D_\alpha} \cap {D_\beta}^\circ = \empt$ holds for each $\beta < \alpha$. Moreover,
${D_\alpha}^\circ \subs \overline{U}$ implies ${D_\alpha}^\circ \cap D_\beta = \empt$
for any $\beta < \alpha$ with $i_\beta = 1$, which together with $D_\alpha \cap {D_\alpha}^\circ = \empt$
yields $$(E_{\alpha,0} \cup {D_\alpha}^\circ) \cap (E_{\alpha,1} \cup D_\alpha) = E_{\alpha+1,0} \cap E_{\alpha+1,1} = \empt.$$
Of course, if $j = 1$ then we shall have $i_\alpha = 0$.

After having completed the inductive construction,
it is trivial to conclude that $E_{\kappa,0}$ and $E_{\kappa,1}$ are two disjoint dense subsets of $X$.
\end{proof}

We shall use lemma \ref{lm:2res} in the proof of our main result, in the induction step of a procedure
where we move from $n$-resolvability to $n+1$-resolvability.

In our following preliminary result, rather than the extent $\ext(X)$, its ``hat" version $\widehat \ext(X)$ will appear.
We recall that $\widehat \ext(X)$ is defined as the smallest cardinal $\lambda$ such that $X$ has no
closed-and-discrete subset of size $\lambda$. Thus we clearly have $\widehat \ext(X) \le \ext(X)^+$,
moreover $\widehat \ext(X) \le \kappa$ is simply equivalent with the statement that for every set $A \in [X]^\kappa$
we have $A' \ne \empt$. We start with defining an auxiliary concept that will be needed in this result.

\begin{definition}
Let  $X$ be a   topological space  and $\kappa$ be a cardinal.
A subset $H\subs X$ is  called {\em \appr {{\kappa}} in $X$}
iff there are $\kappa$ many pairwise disjoint
sets $\{X_{\alpha}:{\alpha}<{\kappa}\}\subs \br X;{\kappa};$ such that
\begin{equation}\notag
\forall\,Y \in [X_\alpha]^\kappa\,({Y}^\circ \ne \empt)\, \mbox{ and }\,  H=\accu{(X_{\alpha})}
\end{equation}
hold true for all ${\alpha}<{\kappa}$.
\end{definition}

The following lemma shows that this definition is not empty.

\begin{lemma}\label{lm:small_acc}
Assume that  $\kappa$ is a regular cardinal and $X$ is a space for which
$\widehat \ext(X) \le \kappa$.  Then

\begin{enumerate}[(1)]
\item $A \in [X]^\kappa$ and $|A'|<{\kappa}$ imply $\accu A\ne \empt$;
\item
if $A \in [X]^\kappa$ is such that $|\accu A|<{\kappa}$ and
$\accu B\ne \empt$ for all $B\in \br A;{\kappa};$
then there is a subset $H \subs A^\circ$ that is \appr {\kappa} in $X$.
\end{enumerate}
\end{lemma}

\begin{proof}
(1) Every point $x\in A'\setm \accu A$
has an open neighbourhood $U_x$ such that $|U_x\cap A|<{\kappa}$.
Then $U=\bigcup\{U_x:x\in A' \setm \accu A\}$ covers $A'\setm \accu A$ and $|A\cap U|<{\kappa}$
because  $\kappa$ is regular.  So we have $|A\setm U| = {\kappa}$
and hence $\empt \ne (A\setm U){}'=\accu{A}$ by
$\kappa \ge \widehat \ext(X)$.

\medskip
\noindent (2)
We start by fixing $\kappa$ pairwise disjoint sets
$\{A_\alpha : \alpha < \kappa \} \subs [A]^\kappa$ and for any
${\alpha}<{\beta}<{\kappa}$ we write
\begin{equation}\notag
A_{{\alpha},{\beta}}=\bigcup_{{\alpha} \le \nu <\beta}A_\nu.
\end{equation}
For fixed ${\alpha}<{\kappa}$ the sequence
\begin{equation}\notag
\{\accu{A_{{\alpha},{\beta}}}:\beta \in {\kappa} \setm \alpha\}
\end{equation}
is increasing and hence
must stabilize since $|A^\circ|< {\kappa}$. This means that there is an ordinal $f({\alpha})<{\kappa}$
such that
\begin{equation}\notag
\accu{A_{{\alpha},{\beta}}}=\accu{A_{{\alpha},f({\alpha})}}
\end{equation}
for all ${\beta} \in {\kappa} \setm f(\alpha)$.
Similarly, the sequence
\begin{equation}\notag
\{\accu{A_{{\alpha},f({\alpha})}}:{\alpha}<{\kappa}\}
\end{equation}
is decreasing and hence it stabilizes: There is an ordinal ${\alpha}^*<{\kappa}$
such that
\begin{equation}\notag
\accu{A_{{\alpha^*},f({\alpha^*})}}=\accu{A_{{\alpha},f({\alpha})}}
\end{equation}
whenever ${\alpha}^* \le {\alpha}<{\kappa}$. We claim that the set
$H=\accu{A_{{\alpha^*},f({\alpha^*})}}$ is
\appr{{\kappa}}  in $X$.

To see this, choose $I\in \br {\kappa}\setm {\alpha}^*; {\kappa};$
in such a way that for any ${\alpha},{\beta}\in I$ with $\alpha < \beta$ we have $f({\alpha})<{\beta}$.
This is possible because $\kappa$ is regular. Then the sets
\begin{equation}\notag
\{A_{{\alpha},f({\alpha})}:{\alpha}\in I\}\subs \br X;{\kappa};
\end{equation}
are pairwise disjoint and, by definition, for all ${\alpha}\in I$ we have both
\begin{equation}\notag
\forall\,B \in [A_{\alpha,f(\alpha)}]^\kappa\,({B}^\circ \ne \empt)\, \mbox{ and }\,\accu{A_{{\alpha},f({\alpha})}}=H.
\end{equation}
\end{proof}

Our next resolvability result uses \appr {\kappa} sets in an essential way.

\begin{theorem}\label{tm:app}
Assume that  $X$ is a space, ${\kappa}=|X|$ is a regular cardinal,
 moreover $\mathcal{H}$ is a disjoint
family of sets \appr {\kappa} in $X$ such that $\bigcup \mathcal{H}$
is dense in $X$. Then the space $X$ is $\kappa$-resolvable.
\end{theorem}

\begin{proof}
For each $H \in \mathcal{H}$ let us fix a disjoint family $\{A_{H,\alpha} : \alpha < \kappa \} \subs [X]^\kappa$
which witnesses that $H$ is \appr {\kappa} in $X$, i.e.
$$\forall\,B \in [A_{H,\alpha}]^\kappa\,({B}^\circ \ne \empt)\, \mbox{ and }\,{A_{H,\alpha}}^\circ = H$$ for all $\alpha < \kappa$.

Note that if $H,K \in \mathcal{H}$ and $\alpha, \beta \in \kappa$ with $\langle H,\alpha \rangle  \ne \langle K,\beta \rangle$ then we have
$|A_{H,\alpha} \cap A_{K,\beta}| < \kappa$. Indeed, if $H = K$ then $A_{H,\alpha} \cap A_{K,\beta} = \empt$. And
if $H \ne K$ then  $|A_{H,\alpha} \cap A_{K,\beta}| = \kappa$ would imply
$\empt \ne (A_{H,\alpha} \cap A_{K,\beta})^\circ \subs H \cap K$, contradicting $H \cap K = \empt$.
This means that the family $$\mathcal{A} = \{A_{H,\alpha} : H \in \mathcal{H},\,\alpha < \kappa  \} \subs [X]^\kappa$$
is {\em almost disjoint}.

But $|\mathcal{H}| \le |X| = \kappa$ implies $|\mathcal{A}| = \kappa$, and then by the regularity of $\kappa$
it follows that $\mathcal{A}$ is also {\em essentially disjoint}. In other words, this means that for every
pair $\langle H,\alpha \rangle$ there is a set $F_{H,\alpha} \in [A_{H,\alpha}]^{<\kappa}$ such that the
collection $$\{B_{H,\alpha} = A_{H,\alpha} \setm F_{H,\alpha} : \langle H,\alpha \rangle \in \mathcal{H} \times \kappa  \}$$
is already disjoint. Note also that for each $\langle H,\alpha \rangle \in \mathcal{H} \times \kappa$ we have
$${B_{H,\alpha}}^\circ = {A_{H,\alpha}}^\circ = H\,.$$

We claim that for every $\alpha < \kappa$ the set $$D_\alpha = \bigcup \{B_{H,\alpha} : H \in \mathcal{H} \}$$
is dense in $X$. Indeed, for any $U \in \tau^+(X)$ there is a set $H \in \mathcal{H}$ with
$H \cap U \ne \empt$, so we may pick a point $x \in H \cap U$. But then $x \in H = B_{H,\alpha}^\circ$ implies
$|U \cap B_{H,\alpha}| = \kappa$, consequently $U \cap D_\alpha \ne \empt$. Thus $\{D_\alpha : \alpha < \kappa \}$
is a family of $\kappa$ many pairwise disjoint dense sets in $X$.
\end{proof}

>From lemma \ref{lm:small_acc} and theorem \ref{tm:app} we may immediately deduce the following corollary that
will be needed in the proof of theorem \ref{tm:gen_e_res}. Maybe ironically, this does not even mention
\appr {\kappa} sets, but its proof does.

\begin{corollary}\label{A'}
Assume that  $\kappa$ is a regular cardinal and $X$ is a space for which
$$\widehat \ext(X) \le |X| = \kappa.$$ If the family $$\mathcal{A}' = \{A' : A \in [X]^\kappa\, \mbox{ and }\, |A'| < \kappa\}$$
is a $\pi$-network in $X$ then the space $X$ is $\kappa$-resolvable.
\end{corollary}

\begin{proof}
By lemma \ref{lm:small_acc} every member $A'$ of $\mathcal{A}'$ includes a set that is \appr {\kappa} in $X$,
hence if $\mathcal{A}'$ is a $\pi$-network in $X$ then so is the family $\mathcal{G}$ of all the sets
that are \appr {\kappa} in $X$.
But then the union of any maximal disjoint subfamily $\mathcal{H} \subs \mathcal{G}$ is clearly dense in $X$,
hence all the assumptions of theorem \ref{tm:app} are satisfied.
\end{proof}

We now turn to another circle of preliminary results that will be used in the proof of our
main theorem \ref{tm:gen_e_res}. Again, we have to start with some definitions.

\begin{definition}\label{df:kappast_closed}
If $X$ is a topological space,
$Y\subs X$, and $\kappa$ is an infinite cardinal, then we call
\begin{equation}\notag
\leka{Y}=\bigcup\{\overline{S}:S\in \br Y;<\kappa;\}
\end{equation}
the $<\,\kappa$-closure of $Y$ in $X$.
We say that $Y$ is  {\em $<{\kappa}$-closed} in $X$ iff $Y=\leka Y$.
If $\kappa = \mu^+$ then instead of $<\,\mu^+$-closure (resp. $<{\mu^+}$-closed)
we simply say $\mu$-closure (resp. $\mu$-closed).
\end{definition}

\begin{definition}\label{df:chain_decomposition}
A {\em chain decomposition of length $\beta$} (for some ordinal $\beta$)
of a set $X$ is an {\em increasing and continuous} sequence
$s = \<X_{\alpha}:{\alpha}<{\beta}\>\,$ such that
$X=\bigcup\{X_{\alpha}:{\alpha}<{\beta}\}$.
(Continuity means that $X_\delta = \bigcup\{X_{\alpha}:{\alpha}<\delta\}$
holds for any limit ordinal $\delta < \beta$. Since we also consider $0$ a limit
ordinal, this implies $X_0 = \empt$.)
\end{definition}

Clearly, if $s = \<X_{\alpha}:{\alpha}<{\beta}\>\,$ is a chain decomposition of $X$
and $Y \subs X$ then  $s\upharpoonright Y = \<Y \cap X_{\alpha}:{\alpha}<{\beta}\>\,$
is a chain decomposition of $Y$. Moreover, if $C \subs \beta$ is a cub (closed and unbounded)
subset of $\beta$ and $C = \{\gamma_i : i < \delta \}$ is the increasing enumeration of $C$
then $s[C] = \< X_{\gamma_i} : i < \delta\>$ is again a chain decomposition of $X$.

\begin{lemma}\label{lm:decomp3}
Assume that $\kappa=\cf(\kappa)\le {\lambda}$ are infinite cardinals
and $X$ is a topological space with $|X|={\lambda}$.
Then $X$ has a chain decomposition $s = \<X_{\alpha}:{\alpha}< \cf(\lambda)\>\,$ such that
$\{X_{\alpha}:{\alpha}<{\cf({\lambda})}\}\subs \br X;<{{\lambda}};$, moreover
\begin{enumerate}[(C1)]
\item $\,\,X_{\alpha}\cap \leka {X\setm X_{\alpha}}\subs
\leka {X_{\alpha+1}\setm X_{\alpha}}$ for each ${\alpha}<{\cf({\lambda})}$.
\end{enumerate}

\medskip

If, in addition, $X$ is $\pi$-regular and not $\cf({\lambda})$-resolvable
then there are a cub set $C \subs \cf(\lambda)$ and a regular closed set
$Y\in \rcp{X}$ such that we also have
\begin{enumerate}[(C2)]
\item $\,\,Y\setm X_{\alpha}$ is $< \kappa$-closed for each ${\alpha} \in C$.
\end{enumerate}

Hence $Y$ has a chain-decomposition
$\{Y_{\alpha}:{\alpha}<{\cf({\lambda})}\}\subs \br Y;<{{\lambda}};$
such that
\begin{equation}\label{C2}
\text{$\,\,Y\setm Y_{\alpha}$ is $< \kappa$-closed for all ${\alpha}<{\kappa}$ }
\end{equation}

\end{lemma}

\begin{proof}
Let us consider first every pair $\<x,S\>$ such that $x \notin S$ but $x\in \leka S$ and assign
to this pair $\<x,S\>$ a set $A(x,S)\in \br S;<\kappa;$ with $x\in \overline {A(x,S)}$.
Moreover, choose a chain decomposition  $\{Z_{\alpha}:{\alpha}<\cf({\lambda})\}\subs \br X;<{{\lambda}};$
of $X$ in an arbitrary manner.
Then we define the sequence $\<X_{\alpha}:{\alpha}<\cf({\lambda})\>$
by transfinite recursion on $\alpha$  as follows:
\begin{enumerate}[(i)]
\item $X_0=\empt$;
\item $X_{\alpha} = \bigcup \{X_{\beta}:{\beta} < \alpha \}$ if $\alpha > 0$ is limit;
\item $X_{{\alpha}+1}= X_{\alpha} \cup Z_\alpha \cup\bigcup\{A(x,X \setm X_\alpha):
x\in X_{\alpha} \cap \leka{X\setm X_{\alpha}}\}.$

\end{enumerate}
Since ${\kappa}\le {\lambda}$ is regular, we can show by an easy transfinite induction that
$|X_{\alpha}|<{\lambda}$ for all ${\alpha}<\cf ({\lambda})$, moreover condition (C1)
obviously follows from case (iii) of our definition. This proves the first half of the lemma.

Now assume that, in addition,  $X$ is not $\cf(\lambda)$-resolvable.
For each $A\subs \cf({\lambda})$ let us set
\begin{equation}\notag
R_A=\bigcup_{{\alpha}\in A}(X_{\alpha+1}\setm X_{\alpha}).
\end{equation}
If $R_A$ would be dense in $X$ for all $A \in [\cf(\lambda)]^{\cf(\lambda)}$
then clearly $X$ would be $\cf({\lambda})$-resolvable,
hence there is a cofinal $A \subs \cf(\lambda)$ and an open set $U \in \tau^+(X)$
with $U \cap R_A=\empt$.

We claim that then for every closed set $F \subs U$ and for every ${\alpha}\in A$
we have that $F\setm X_{\alpha}$ is $<\kappa$-closed.
Indeed, assume on the contrary that
$x\in \leka {F\setm X_{\alpha}}\cap X_{{\alpha}}$.
Then, by (C1), there is a set
$S\in \br X_{\alpha+1}\setm X_{\alpha};<\kappa;$
with $x\in \overline S$. Since $x\in F \subs U$, this implies $U\cap S \ne \empt$,
which contradicts $U \cap R_A=\empt\,,$ as ${\alpha}\in A$ and $S \subs X_{\alpha+1} \setm X_\alpha$.

Now, if $X$ is also $\pi$-regular then there is a regular closed set
$Y\in \rcp X$ with $Y \subs U$.
Let us consider the set
\begin{displaymath}
C=\{{\alpha}<{\cf({\lambda})}: Y\setm X_{\alpha}\text{ is $<{\kappa}$-closed}\}.
\end{displaymath}
$C$ is clearly closed in $\cf(\lambda)$ and $A\subs C$ by the above,
hence $C$ is cub in $\cf(\lambda)$.
This completes the proof of lemma \ref{lm:decomp3}.

\end{proof}

We have one more preparatory result involving chain decompositions that will
be used in the proof of our main theorem.

\begin{lemma}\label{lm:star}
Assume that $Y$ is a $\pi$-regular space that is not ${\omega}$-resolvable. Then
for every chain decomposition $\{Y_{\alpha}:{\alpha}<\mu\}$  of $Y$ there are  $T\in \rcp Y$
and a dense subset $Z \subs T$ such that
\begin{equation}\label{eq:init_closed}
\text{$\overline{Y_{\alpha}\cap Z}\subs  Y_{\alpha}$ \mbox{ for all }
${\alpha}<{\mu}$}.
\end{equation}
\end{lemma}

\begin{proof} By the continuity of chain decompositions,
for each point $x\in Y$ there is a unique ordinal   ${\alpha}(x) < \mu$
such that $$x\in Y_{{\alpha}(x)+1}\setm Y_{{\alpha}(x)}.$$
For any subset $A\subs Y$ let us define
\begin{displaymath}
 A^*=\{x\in A: x \notin \overline{A\cap Y_{{\alpha}(x)}}\}.
\end{displaymath}
We claim that $A^*$ is dense in $A$ for every $A\subs Y$.

Indeed, if  $U$ is open and  $U \cap A \ne \empt$, then
pick  $a\in U \cap A$ such that $\alpha(a)$ is minimal. Then
$a\in A^*$ because by the minimality of $\alpha(a)$ we have $U\cap Y_{{\alpha}(a)}=\empt$.

Let us now define the sets  $\{D_j : j < \omega \}$
by means of the following recursive formula:
\begin{displaymath}
 D_j=\big(Y\setm \bigcup_{i<j}D_i\big)^*.
\end{displaymath}

The pairwise disjoint sets  $\{ D_j : j < \omega\}$ cannot be all dense in $Y$
because $Y$ is not ${\omega}$-resolvable, but $D_0 = Y^* $ is dense.
So there is $m\in {\omega}$ such that
$D_m$ is dense but $D_{m+1}$ is not, hence
$U \cap D_{m+1} = \empt$ for some
$U \in \tau^+(X)$. Now, pick  $T\in \rcp Y$ with $T\subs U$.

Then $U \cap D_{m+1} = U \cap \big(Y\setm \bigcup_{i\le m}D_i\big)^*=\empt$ implies

\begin{equation}\label{eq:tcovered}
 T\subs U \subs \bigcup_{j\le m}D_j\,,
\end{equation}
and clearly $Z=T\cap D_m$ is dense in $T$.

Now it remains to show that
\begin{equation}\label{eq:Dm1}
\text{ $\overline{Z \cap Y_{\alpha}}\subs Y_{\alpha}$
for all ${\alpha}<{\mu}$.}
\end{equation}
To see this, fix $\alpha < \mu$ and consider first any point $x\in T$.
Then $x\in D_j$ for some $j\le m$ by (\ref{eq:tcovered}).
This means that  $x\in (Y\setm \bigcup_{i<j}D_i)^*$, i.e.
$x\notin \overline{(Y\setm \bigcup_{i<j}D_i)\cap Y_{{\alpha}(x)}}$.
But $D_m\subs Y\setm \bigcup_{i<j}D_i$, hence
we have $x \notin \overline{D_m \cap Y_{\alpha(x)}}$.

On the other hand, for every point $x \in \overline{Z \cap Y_\alpha}\, \big(\subs T\big)$ we have
$x \in \overline{D_m \cap Y_\alpha}$ because $Z \subs D_m$. This together
with $x \notin \overline{D_m \cap Y_{\alpha(x)}}$ implies
$\alpha(x) < \alpha$ because the sets $Y_\beta$ are increasing. So,by the definition of $\alpha(x)$ we have
$x \in Y_{\alpha(x)+1} \subs Y_\alpha$, and this means that $\overline{Z \cap Y_{\alpha}}\subs Y_{\alpha}$.

\end{proof}

Our next preliminary results will be used in the proof of theorem \ref{tm:steppingup}, a stepping-up result
concerning resolvability of certain spaces.

\begin{lemma}\label{lm:many_intersection}
Assume that ${\kappa}$ is an infinite cardinal,
$X$ a topological space, and we have a disjoint subfamily $\mc H \subs \mathfrak{R}_\kappa(X)$ such that
for each $U\in \tau^+(X)$ $$|\{H\in \mc H: H\cap U\ne \empt\}|=\kpl.$$
Then $X$ is $\kpl$-resolvable.
\end{lemma}

\begin{proof}[Proof of lemma \ref{lm:many_intersection}]
Obviously, $|\mc H|=\kpl$, so we can fix a one-one
enumeration $\mc H=\{H_{\xi}:{\xi}<\kpl\}$
of $\mc H$. Every $H_{\xi}$ is ${\kappa}$-resolvable, and so has a partition
\begin{equation}
 H_{\xi}=\bigcup{}^*\{H^i_{\xi}:i<{\xi}\}
\end{equation}
into dense subsets.
Then for every $i<\kpl$ the set
\begin{equation}
D^i=\bigcup\{H^i_{\xi}:i<{\xi}<\kpl\}.
\end{equation}
is dense in $X$.
Indeed, for each $U \in \tau^+(X)$
by our assumption there is ${\xi}>i$ with $U\cap H_{\xi}\ne \empt$.
But $H^i_{\xi}$ is dense in $H_{\xi}$,
so we have $U \cap D^i \supset U \cap H^i_{\xi} \ne \empt$ as well.
As the dense sets $\{D^i : i < \kpl \}$ are pairwise disjoint, our proof is complete.

\end{proof}

To formulate the following corollary of lemma \ref{lm:many_intersection}, we need
one more definition.

\begin{definition}\label{df:nice}
Let $X$ be any topological space and $\kappa$ an infinite cardinal. A
(necessarily closed) subset $F \subs X$ is called {\em $\kappa$-nice in $X$}
if there is a disjoint family $\{A_\alpha : \alpha < \kappa^+ \} \subs \mathfrak{R}_\kappa(X)$
such that $$F = \bigcap_{\alpha < \kappa^+} \overline{\bigcup \{A_\beta  : \beta \in \kappa^+ \setm \alpha \}} \,.$$
\end{definition}

Following the terminology of \cite{JSz}, we call a space $\lambda$-compact
if every subset of it of cardinality $\lambda$ has a complete accumulation point.

\begin{corollary}\label{co:nice}
Let $\kappa$ be an infinite cardinal and $X$ be a $\kappa^+$-compact space.
If there is a disjoint family $\mathcal{F}$ of both $\kappa$-resolvable
and $\kappa$-nice subsets of $X$
such that $|\mathcal{F}| \le \kappa^+$ and $\,\bigcup \mathcal{F}$ is dense in $X$,
then $X$ has a $\kappa^+$-resolvable open subset.
\end{corollary}

\begin{proof}
If for every $U\in \tau^+(X)$ we have $|\{F\in \mc F: F \cap U \ne \empt\}|=\kpl$
then $X$ itself is $\kappa^+$-resolvable by lemma \ref{lm:many_intersection}.
So assume that $U\in \tau^+(X)$ is such that $\mathfrak{F}^* = \{F \in \mathcal{F} : U \cap F \ne \empt \}$
has cardinality $\le \kappa$.

For each $F \in \mathcal{F}^*$ let us fix
a disjoint family $$\{A^F_\alpha : \alpha < \kappa^+ \} \subs \mathfrak{R}_\kappa(X)$$
witnessing that $F$ is nice, as required in definition \ref{df:nice}.
We claim that for every pair $\{F,G\} \in [\mathcal{F^*}]^2$ there is an $\alpha = \alpha(F,G) < \kappa^+$ such that
$$\bigcup \{A^F_\beta  : \beta \in \kappa^+ \setm \alpha \} \cap \bigcup \{A^G_\beta  : \beta \in \kappa^+ \setm \alpha \} = \empt\,.$$
Indeed, otherwise we could select a set $I \in [\kappa^+]^{\kappa^+}$ and distinct points $\{x_\alpha : \alpha \in I\}$ such that
$$x_\alpha \in \bigcup \{A^F_\beta  : \beta \in \kappa^+ \setm \alpha \} \cap \bigcup \{A^G_\beta  : \beta \in \kappa^+ \setm \alpha\}$$
whenever $\alpha \in I$. But then $\{x_\alpha : \alpha \in I \}^\circ \ne \empt$ would be a subset of
$$
\bigcap_{\alpha \in I} \overline{\bigcup \{A^F_\beta  : \beta \in \kappa^+ \setm \alpha \}} \cap
\bigcap_{\alpha \in I} \overline{\bigcup \{A^G_\beta  : \beta \in \kappa^+ \setm \alpha \}}
= F \cap G\,,$$ contradicting $F \cap G = \empt$.

Now, $|\mathcal{F}^*| \le \kappa$ implies that there is an ordinal $\gamma < \kappa^+$
such that $\alpha(F,G) < \gamma$ for all pairs $\{F,G\} \in [{\mc F}^*]^2$. Consequently,
the elements of the family $$\mc J = \{A^F_\alpha : F \in \mathcal{F}^* \mbox{ and }
\alpha \in \kappa^+ \setm \gamma\} \subs \mathfrak{R}_\kappa(X)$$
are pairwise disjoint and, by our assumptions, both $\bigcup \mathcal{F}^* \cap U$ and 
$\bigcup \mc J \cap U$ are dense in $U$.

Thus, for every $V \in \tau^+(U)$ there is $F\in {\mc F}^*$ for which $V \cap F \ne \empt$.
But this clearly implies that $|\{\alpha \in \kappa^+ \setm \gamma : V \cap A^F_\alpha \ne \empt \}| = \kappa^+$,
hence $U$ and the family $$\mc H = \mc J \upharpoonright U = \{U \cap A^F_\alpha : F \in \mathcal{F}^* \mbox{ and }
\alpha \in \kappa^+ \setm \gamma\}$$ satisfy the assumptions of lemma
\ref{lm:many_intersection}, consequently $U$ is $\kappa^+$-resolvable.

\end{proof}

\newpage

\section{The Main Result}

We are now ready to formulate and prove our main result.

\begin{theorem}\label{tm:gen_e_res}
Let $\kappa$ be an uncountable regular cardinal.
Then every regular space $X$ satisfying $$\Delta(X) \ge \kappa \ge \widehat \ext(X)$$
is $\omega$-resolvable. Consequently, every regular space $X$ that satisfies
$\Delta(X) > \ext(X)$ is $\omega$-resolvable.
\end{theorem}

\begin{proof}
For convenience, after fixing $\kappa$, we denote by $\mathcal{C}$ the class of
all regular spaces $X$ that satisfy $\Delta(X) \ge \kappa \ge \widehat \ext(X)$.
Clearly, the class $\mathcal{C}$ is regular closed hereditary, that is for every
$X \in \mathcal{C}$ we have $\rcp X \subs \mathcal{C}$. By corollary \ref{cl:rc},
to prove that all members of $\mathcal{C}$ are $\omega$-resolvable it suffices to
show that every $X \in \mathcal{C}$ with $|X| = \Delta(X)$ is $\omega$-resolvable.

To achieve this, for any cardinal ${\lambda}\ge {\kappa}$ we set
$$\clas \lambda = \{X \in \mathcal{C} : |X| = \Delta(X) = \lambda\}\,,$$
and then we prove by induction on $\lambda \ge \kappa$ that

\smallskip

\begin{enumerate}[($*_{\lambda}$)]
 \item every member of $ \clas {\lambda}$ is ${\omega}$-resolvable.
\end{enumerate}

\smallskip

So let us assume now that $\lambda \ge \kappa$ and $(*_\mu)$ holds whenever
$\kappa \le \mu < \lambda$. Clearly, this implies that every space $X \in \mathcal{C}$
with $|X| < \lambda$ is $\omega$-resolvable.

To deduce $(*_\lambda)$ from this,
by theorem \ref{tm:illanes}, it suffices to show that
every member of $\clas {\lambda}$ is $n$-resolvable for all $n\in {\omega} \setm \{0\}$.
This, in turn, will be proved by a subinduction on $n\in {\omega} \setm \{0\}$.
Therefore we assume from here on that for some $n > 0$ we have
\begin{enumerate}[($\circ_n$)]
\item  every member of $ \clas {\lambda}$ is $n$-resolvable
\end{enumerate}
and we want to deduce ($\circ_{n+1}$) from this.
(Of course, ($\circ_1$) holds trivially.)

To prove ($\circ_{n+1}$), we observe first that the class
$\clas {\lambda}$ is also regular closed hereditary,
hence by corollary \ref{cl:rc} again, ($\circ_{n+1}$)
is implied by the following seemingly weaker statement:
\begin{enumerate}[($\circ'_{n+1}$)]
\item  every member of $\clas {\lambda}$ has an $(n+1)$-resolvable subspace.
\end{enumerate}

Now, the proof of ($\circ'_{n+1}$) branches into two: Namely, the initial case
$\lambda = \kappa$ and the case $\lambda > \kappa$ of the induction on $\lambda$ 
are handled differently.

\medskip\noindent{\bf Case 1. }{ ${\lambda}={\kappa}$}
\medskip

Consider any $X \in \clas \kappa$ and recall that our aim is to show that $X$ has an $(n+1)$-resolvable subspace.
If $X$ is ${\kappa}$-resolvable then we are done. Otherwise
by lemma \ref{lm:small_acc} there is $Q \in \rcp X$ such that
\begin{equation}\label{eq:A1}
\text{ $|A'|={\kappa}$ for all $A\in \br Q;{\kappa};$.}
\end{equation}

It easily follows from (\ref{eq:A1}) that for every $Y\in \br Q;<{\kappa};$ and for every
$B\in \rcp Q (\subs \rcp X)$ we have $B \setm Y \in \clas \kappa$. Consequently, if $Y \cap B$
is dense in $B$ then $B$ is $(n+1)$-resolvable because $B \setm Y$ is $n$-resolvable by ($\circ_n$).
So from here on we can assume that
\begin{equation}\label{eq:small_nowhere}
\text{every  set in $\br Q;<{\kappa};$ is nowhere dense.}
\end{equation}

Let us now apply lemma \ref{lm:decomp3} to the space $Q$ and with the choice $\kappa = \lambda$.
This yields us some $Y\in \rcp Q$ and a chain decomposition
$\{Y_{\alpha}:{\alpha}<{{\kappa}}\}\subs \br Y;<{{\kappa}};$ of $Y$ such that
\begin{equation}\label{C22}
\text{$Y\setm Y_{\alpha}$ is $<\kappa$-closed
for each ${\alpha}<{\kappa}$.}
\end{equation}


If $Y$ happens to be $\omega$-resolvable (or just $(n+1)$-resolvable) then, of course, we are done.
Otherwise we may apply lemma \ref{lm:star} to the chain decomposition
$\{Y_{\alpha}:{\alpha}<{{\kappa}}\}$ of $Y$
to obtain  $T\in \rcp Y$
with a  dense subset $Z \subs T$ such that
\begin{equation}\label{eq:init_closed2}
\text{$\overline{Y_{\alpha}\cap Z}\subs  Y_{\alpha}$ for all
${\alpha}<{\kappa}$}.
\end{equation}

Write $R_{\alpha}=Y_{\alpha+1}\setm Y_{\alpha}$ for ${\alpha}<{\kappa}$.
For each $x\in Y$  we let  ${\alpha}(x)\in {{\kappa}}$
be the unique ordinal with $x\in R_{{\alpha}(x)}$.
We call a subset $E\subs Y$  {\em rare} iff $|E\cap R_{\alpha}|\le 1$
for all ${\alpha}<{{\kappa}}$. It is immediate from (\ref{C22}) and (\ref{eq:init_closed2})
that every rare subset $E$ of $Z$ of size $< \kappa$ is closed-and-discrete, i.e. satisfies
$E' = \empt$.

Let us now consider the family
\begin{equation}\notag
\mc D=\{D\in \br Z;{\kappa}; :  D \text{ is discrete and } \,\forall \,E \in [D]^{< \kappa}\, (E' = \empt)  \}.
\end{equation}
The derived set $E'$ of $E$ above is always meant to be taken in $T$ (or equivalently in $X$), not in $Z$.
It is obvious from the definition that for every $D \in \mc D$ we have $D' = D^\circ$ and
$[D]^\kappa \subs \mc D$.

\begin{claim}
\label{D-kör}
For every $D \in \mc D$ we have $\Delta(D^\circ) = \kappa$, consequently $D^\circ \in \clas \kappa$.
\end{claim}

\begin{proof}[Proof of the claim]
Assume that $G$ is any open set with $G \cap D^\circ \ne \empt$ and pick
a point $x \in G \cap D^\circ$. By the regularity of the space $X$ there is an open set $H$ such that
$x \in H \subs \overline{H} \subs G$. Then we have $|H \cap D| = \kappa$, as $x \in D^\circ$, and hence
$|(H \cap D)^\circ| = \kappa$ by (\ref{eq:A1}). But we clearly have
$$(H \cap D)^\circ \subs \overline{H} \cap D^\circ \subs G \cap D^\circ,$$
hence $|G \cap D^\circ| = \kappa$ and therefore $\Delta(D^\circ) = \kappa$.
Since $D^\circ$ is closed in $X$ it is obvious that $\widehat{\ext}(D^\circ) \le \widehat{\ext}(X) \le \kappa$
and hence $D^\circ \in \clas \kappa$.
\end{proof}

We note that this proof used the full force of the regularity of our space.

\begin{claim}
\label{eq:pi-net}
Assume that $V\in \rcp T$ and the set $S\subs V\cap Z$ is dense in $V$.
If $S$ is not ${\kappa}$-resolvable then there is some $D \in \mc D$ 
such that $D \subs S$.
\end{claim}

\begin{proof}[Proof of the claim]
Since every member of $\tau^+(S)$ is somewhere dense in $V$ and hence in $Q$,
it follows from (\ref{eq:small_nowhere}) that $\Delta(S) = \kappa = \cf(\kappa) > \omega$.
If $S$ is not ${\kappa}$-resolvable then  \cite[Theorem 2.2]{JSSz_res2} implies that $S$ must have a
(relatively) discrete subset $J$ of size $\kappa$. Clearly, there is $D \in [J]^\kappa$ that is rare.
But then $D\in \mc D$ because, as was pointed out above, we have $E' = \empt$ for all rare sets
$E$ of size $< \kappa$.
\end{proof}

By the (sub)inductive assumption ($\circ_n$) we have a partition
$$T=\bigcup_{i<n}Z_i$$ of $T$ into $n$ pairwise disjoint dense subsets $Z_i$.
Since $Z \subs T$ is also dense, it is not possible that $Z \cap Z_i$
is nowhere dense for all $i < n$. Thus we can assume,
without loss of generality,  that $Z\cap Z_0$ is somewhere dense, say it is dense in
$V\in \rcp T$.

If there is some $W \in \rcp V$ for which $W \cap Z\cap Z_0$ is $\kappa$-resolvable
(or just ($n+1$)-resolvable) then again we are done. Otherwise, by claim \ref{eq:pi-net},
for each $W \in \rcp V$ the set $W \cap Z\cap Z_0$ includes a member of $\mathcal{D}$ ,
hence we may assume that
\begin{equation}\label{eq:pinet}
\text{$\mc E = \{D \in \mc D : D \subs V \cap Z \cap Z_0\}$ is a $\pi$-network in $V$}.
\end{equation}

Now we distinguish two subcases.

\noindent{ \bf Subcase 1.} $$\mc E_0 = \{D\in \mc E :
\accu D\cap Z\cap Z_0 =\empt \}$$
is a $\pi$-network in $V$.

\smallskip

In this case the family $\mc F = \{D^\circ : D \in \mc E_0 \}$ is also
a $\pi$-network in $V$ because $V$ is $\pi$-regular,
hence by definition $\mc F$ is a $\pi$-network in $V\setm (Z\cap Z_0)$ as well.
But every $D^\circ \in \mc F$ is $n$-resolvable by claim \ref{D-kör} and $(\circ_n)$,
hence so is $V \setm (Z \cap Z_0)$ by theorem \ref{tm:el}. This, however, implies that $V$ is
$(n+1)$-resolvable because $V \cap Z \cap Z_0$ is also dense in $V$.

\medskip

\noindent{ \bf Subcase 2.} $\mc E_0$ is not a $\pi$-network in $V$, i.e. there is  $U\in \rcp V$
such that if $D \in \mc D$ and $D \subs U \cap Z \cap Z_0$ then $D^\circ \cap U \cap Z \cap Z_0 \ne \empt$.
Now $$\mc E_1 =\{D \in \mc D : D \subs U \cap Z \cap Z_0 \} = \{D \in \mc E : D \subs U\}$$
is a $\pi$-network in $U \cap Z \cap Z_0$. From this it follows that the set
$$S = (U \cap Z \cap Z_0) \cap \bigcup \{D^\circ : D \in \mc E_1\}$$ is dense in $U \cap Z \cap Z_0$.

Now, it is easy to check then that the space $U \cap Z \cap Z_0$, the cardinal $\kappa$, the family $\mc E_1$,
and the dense subset $S$ of $(U \cap Z \cap Z_0)$ satisfy all the assumptions of  lemma \ref{lm:2res},
hence $U \cap Z \cap Z_0$, and thus $U \cap Z_0$ as well, is 2-resolvable.
But $U\setm Z_0$ is clearly $n-1$-resolvable, and so it follows that
$U$ is $(n+1)$-resolvable. This completes the proof of ($\circ'_{n+1}$) in
the case ${\lambda}={\kappa}$.

\medskip

\noindent{\bf Case 2.}  ${\lambda}>{\kappa}$.

\medskip

Recall that our aim is to show that every space $X \in \clas \lambda$
has an $(n+1)$-resolvable subspace. Assume first that there are $B \in \rcp X$
and a dense subset $A$ of $B$ with $|A|<\lambda$ such that $B \setm A$ is $\kappa$-closed.
Then we have $\Delta(B \setm A) = \lambda$ because $|A| < \lambda$ and
$\widehat \ext(B\setm A)\le {\kappa}$ because $B\setm A$ is ${\kappa}$-closed,
consequently  $B\setm A \in \clas{\lambda}$. So the (sub)inductive assumption ($\circ_n$)
implies that $B \setm A$ is $n$-resolvable and hence $B$ is  $(n+1)$-resolvable.

Thus we may assume from here on that
\begin{equation}\label{eq:small_nowhere2}
\text{ if $A \in [X]^{< \lambda}$ and
$X\setm A$ is ${\kappa}$-closed then $A$ is nowhere dense.}
\end{equation}

Let us now apply lemma \ref{lm:decomp3} to the space $X$ and the cardinals $\lambda$
and $\kappa^+$. This is possible because $\lambda \ge \kappa^+$.
This way we obtain
$Y\in\rcp X$ with a chain decomposition
$$\{Y_{\alpha}:{\alpha}<{\cf({\lambda})}\}\subs \br Y;<{{\lambda}};$$
of length $\cf(\lambda)$ such that for each $\alpha < \cf(\lambda)$ we have
\begin{equation}\label{C222}
\text{$Y\setm Y_{\alpha}$ is $\kappa$-closed
for each ${\alpha}<{\lambda}$.}
\end{equation}

%
%
Note that then each $Y_{\alpha}$ is nowhere dense by (\ref{eq:small_nowhere2}).

For any point $x \in Y$ we again define the ordinal   ${\alpha}(x) < {\cf({\lambda})}$
by the formula $x\in Y_{{\alpha}(x)+1}\setm Y_{{\alpha}(x)}$ and call a set
$E\subs Y$ is {\em rare} iff
$|E\cap (Y_{{\alpha}+1}\setm Y_{{\alpha}})|\le 1$
for all ${\alpha}<{\cf({\lambda})}$.

If $Y$ is $\omega$-resolvable then we are done. Otherwise we may
apply lemma \ref{lm:star} to obtain
$T\in \rcp Y$ with a dense subset $D \subs T$ such that
\begin{equation}\label{eq:init_closed3}
\text{$\overline{Y_{\alpha}\cap D}\subs  Y_{\alpha}$ for all
${\alpha}<\cf({\lambda})$}.
\end{equation}
 
We claim that $D$ has no rare subset of cardinality $\kappa$. 
This is because for any rare set $E \in [D]^\kappa$ we would have had $E' = \empt$,
contradicting $\widehat{\ext}(X) \le \kappa$. 

To see this, pick any point $x \in Y$. Then
$x \notin \overline{D \cap Y_{\alpha(x)}} \subs Y_{\alpha(x)}$
by \ref{eq:init_closed3}, moreover $x \notin \overline{E \setm Y_{\alpha(x)+1}}$
because $Y \setm Y_{\alpha(x)+1}$ is $\kappa$-closed by (\ref{C222}). 
But $|E \cap (Y_{\alpha(x)+1} \setm Y_{\alpha(x)})| \le 1$, hence we clearly have $x \notin E'$.

Consequently, if we got this far, i.e. no $(n+1)$-resolvable subspace has been found yet,
then we must have $\cf({\lambda}) < {\kappa}$. Indeed, since each $Y_\alpha$
is nowhere dense but $D$ is not, there are cofinally many $\alpha < \cf(\lambda)$
with $D \cap (Y_{\alpha+1} \setm Y_\alpha) \ne \empt$. But then $\cf({\lambda}) \ge {\kappa}$
would clearly imply the existence of a rare subset of $D$ of size ${\kappa}$.

\smallskip

Let us now put $T_{\alpha}= T \cap Y_{\alpha}$ and
$Z_{\alpha}=\overline{(D\cap T_{\alpha+1})}\setm T_\alpha$
for ${\alpha}<\cf({\lambda})$.
Then $Z_{\alpha}\subs T_{\alpha+1}\setm T_{\alpha}$ by (\ref{eq:init_closed3}) and $Z_\alpha$ is $\kappa$-closed,
being the intersection of a closed and a $\kappa$-closed set.
This clearly implies $\widehat{\ext}(Z_\alpha) \le \kappa$.

Moreover, the set
\begin{equation}\notag
 \text{$Z=\bigcup_{{\alpha}<\cf({\lambda})}Z_{\alpha}$ } \subs T
\end{equation}
is dense in $T$ because $D \subs Z$. We also have $\widehat{\ext}(Z) \le \kappa$
because $\widehat{\ext}(Z_\alpha) \le \kappa$ for each $\alpha < \cf(\lambda)$ and $\cf(\lambda) < \kappa = \cf(\kappa)$.

The following observation will be crucial in the rest of our proof.

\begin{claim}\label{cl:zdense}
Every  set $H\in \br Z;\le{\kappa};$ is nowhere dense.
\end{claim}
\begin{proof}[Proof of the claim]
Let us fix $H\in \br Z;\le{\kappa};$ and pick $\alpha < \cf(\lambda)$. Then
we have $\overline{H \cap T_{\alpha}}\subs \overline{D \cap T_\alpha} \subs T_{\alpha}$
by (\ref{eq:init_closed3})
and $\overline{H\setm T_{\alpha+1}}\subs T\setm T_{{\alpha}+1}$ by (\ref{C222}).
Moreover, we also have $$\overline{H \cap (T_{\alpha+1} \setm T_\alpha)} \subs Z_\alpha \subs T_{\alpha+1} \setm T_\alpha$$
because $Z_\alpha \subs T_{\alpha+1} \setm T_\alpha$ is $\kappa$-closed, hence
$$\overline{H}\cap (T_{\alpha+1}\setm T_{\alpha}) = \overline{H \cap (T_{\alpha+1} \setm T_\alpha)} .$$

This then implies that
$\{\overline{H}\cap (T_{\alpha+1}\setm T_{\alpha}):{\alpha}<\cf {\lambda}\}$
is a partition of $\overline{H}$ into relatively clopen subsets of size $< \lambda$.
Consequently, for all $U \in \tau^+(\overline{H})$ we have $ \Delta(U) < \lambda$,
while for every $W \in \tau^+(X)$ we have $\Delta(W) = \lambda$. But this implies that
$\intt\overline{H}=\empt$, i.e. $H$ (or, equivalently $\overline{H}$) is nowhere dense.

\end{proof}

If there are an $\alpha < \cf(\lambda)$ and an $R \in \rcp {Z_\alpha}$ with $\Delta(R) \ge \kappa
( \ge \widehat{\ext}(R))$ then, as $R \in \mathcal{C}$ and
$|R| < \lambda$, our inductive hypothesis implies that  $R$ is
$\omega$-resolvable, hence we are done. 

Consequently, we may assume that
$$\mathcal{P}_\alpha = \{U \in \tau^+(Z_\alpha): |U| < \kappa \}$$
is a $\pi$-base of $Z_\alpha$ for each $\alpha < \cf(\lambda)$.
For any ${\alpha}< \cf({\lambda})$ let
$\mc E_{{\alpha}}$
be a maximal disjoint subfamily of ${\mc P}_\alpha$. It follows then that
$E_{{\alpha}}=\bigcup {\mc E_{{\alpha}}}$ is a dense open subset of $Z_\alpha$,
consequently
\begin{equation}\notag
E=\bigcup_{{\alpha}<\cf({\lambda})}E_{\alpha}
\end{equation}
is dense in $Z$ and hence in $T$.

Let us now put
$F_{\alpha}=Z_{\alpha}\setm E_{\alpha}$ for all ${\alpha}<\cf({\lambda})$ and
\begin{equation}\notag
\text{$F=\bigcup_{{\alpha}<\cf({\lambda})}F_{\alpha}$.}
\end{equation}
Since $F_{\alpha}$ is closed in $Z_{\alpha}$ we have
$\widehat \ext(F_{\alpha})\le \widehat \ext(Z_{\alpha})\le {\kappa}$ and so,
by $\cf(\lambda) < \kappa$,
\begin{equation}\label{eq:fe}
\text{$\widehat \ext(F)\le  {\kappa}$}
\end{equation}
as well.

We claim that $F$ is also dense in $Z$.
Assume on the contrary that  $F\cap V=\empt$ for some $V\in \rcp{Z}$, i.e. $V\subs E$. Then
$V \cap Z_{\alpha}\subs E_{\alpha}$ for each $\alpha < \cf(\lambda)$, hence
${\mc U}_\alpha = \{U \cap (V\cap Z_{\alpha}) : U \in {\mc E_\alpha}\}$
yields a partition of  $V \cap Z_{\alpha}$ into (relatively) clopen subsets of $V \cap Z_\alpha$.
But $V \cap Z_\alpha$ is closed in $Z_\alpha$, consequently,
$\widehat \ext(Z_{\alpha}) \le  {\kappa}$ implies
$|\mc U_{{\alpha}}|<{\kappa}$. But then we also have
$|V\cap Z_{\alpha}|<{\kappa}$ because $|U| < \kappa$ for each $U \in \mc U_\alpha$ and ${\kappa}$ is
a regular cardinal. This, in turn, implies
$|V\cap Z|<{\kappa}$ because $\cf {\lambda}<{\kappa}$. But $V\cap Z$ is somewhere dense and this
contradicts claim \ref{cl:zdense}.
So $F$ is indeed dense in $Z$.

As $F$ is dense in $Z$, applying claim \ref{cl:zdense} again we conclude that
\begin{equation}\label{eq:fd}
\text{$\lambda \ge \Delta(F) > {\kappa}$.}
\end{equation}
Putting (\ref{eq:fe})
and (\ref{eq:fd}) together, our inductive hypotheses, including
$(\circ_n)$, imply that $F$ is $n$-resolvable, hence
$Z$ is $(n+1)$-resolvable because $E\cap F=\empt$.
Thus $(\circ'_{n+1})$ is verified and the proof is completed.
\end{proof}

Let us now make a few comments on the assumptions of our main theorem
\ref{tm:gen_e_res}. Although the uncountability of $\kappa$ was used
in our proof when we referred to theorem 2.2 in \cite{JSSz_res2}, theorem
\ref{tm:gen_e_res} is valid for $\kappa = \omega$ as well. Indeed, to
see this we note that $\widehat \ext (X) = \omega$ just means that $X$ is
countably compact, and Pytkeev proved in \cite{Py} that crowded countably
compact regular spaces are even $\omega_1$-resolvable.

The question if the assumption on the regularity of $\kappa$ is essential is
more interesting and we do not know the answer to it. We only have the
following partial positive result in the case when $\kappa$ is a singular
cardinal of countable cofinality.

\begin{theorem}\label{tm:eomegacof}
Let $\kappa$ be a singular cardinal of countable cofinality.
Then every regular space $X$ satisfying $$\Delta(X) \ge \kappa \ge \widehat \ext(X)$$
is 2-resolvable.
\end{theorem}

\begin{proof}
Using theorem \ref{tm:gen_e_res} and theorem \ref{tm:el} it clearly suffices to
show that any regular space $X$ with $$|X| = \Delta(X) = \widehat \ext(X) = \kappa$$
has a 2-resolvable subspace.

If there is $R \in \rcp X$ with $\ext(R) < \kappa$ then $R$ is $\omega$-resolvable 
by theorem \ref{tm:gen_e_res}, 
hence we may assume that $\widehat \ext(R) = \kappa$ for
all $R \in \rcp X$. Also, if some $G \in \tau^+(X)$ has a dense subset $Y$ of cardinality
$< \kappa$ then $\Delta(G) = \Delta(X) = \kappa$ implies that $G \setm Y$ is also
dense in $G$, hence $G$ is 2-resolvable. Thus we may also assume that every set $Y \in [X]^{<\kappa}$
is nowhere dense in $X$.

By $\cf(\kappa) = \omega$ we can fix a strictly increasing
sequence of cardinals $\<\kappa_n : n < \omega \>$ with $\kappa = \sum \{\kappa_n : n < \omega \}$.
Then we may choose a sequence of sets $\{Y_n:n<{\omega}\}\subs \br X;<{\kappa};$ with $|Y_n| = \kappa_n$ and
$\bigcup_{n\in {\omega}}Y_n=X$. Each $\overline{Y_n}$ is nowhere dense,
hence we may define by a straightforward induction a sequence
$\{U_n : n < \omega\} \subs \tau^+(X)$ such that $\overline{U_{n+1}} \subsetneqq U_n$
for all $n < \omega$, moreover $\bigcap\{U_n : n < \omega \} = \empt$.

Then $R_n = \overline{U_n} \setm \intt(\overline{U_{n+1}}) \in \rcp X$ and
$\{R_{2n} : n < \omega \}$ is clearly a discrete collection in $X$. But
$\widehat \ext(R_{2n}) = \kappa$ implies the existence of a set
$D_n \in [R_{2n}]^{\kappa_n}$ that is closed discrete in $R_{2n}$ and
hence in $X$. Consequently, $$D = \bigcup \{D_n : n < \omega \} \in [X]^\kappa$$ is also
closed discrete in $X$, contradicting $\widehat \ext(X) = \kappa$ and
completing the proof.

\end{proof}

\section{Stepping-up resolvability}

Let $\kappa$ be an infinite cardinal and denote by $\mc L_\kappa$ the class of all
regular spaces $X$ that are $\kpl$-compact and satisfy $|X| = \Delta(X) = \kpl$.
(We recall that the $\kpl$-compactness of $X$ means that $A^\circ \ne \empt$ for
each $A \in [X]^{\kpl}$.) Our aim is then to prove the following stepping up
result.

\begin{theorem}\label{tm:steppingup}
If every member of $\mc L_\kappa$ is $\kappa$-resolvable then actually
every member of $\mc L_\kappa$ is $\kappa^+$-resolvable.
\end{theorem}

Before giving the proof of this, however, we have to formulate and prove the
following lemma.

\begin{lemma}
Assume that $X \in \mc L_\kappa$ has no \appr {{\kappa^+}} subset. Then

\begin{enumerate}[(i)]\label{lm:L_k}
\item for any $A \in [X]^{\kappa^+}$ we have $A^\circ \in \mc L_\kappa$;

\medskip

\item there are $\kpl$ many pairwise disjoint sets of the form $A^\circ$
with $A \in [X]^{\kpl}$.
\end{enumerate}
\end{lemma}

\begin{proof}
(i) It is immediate from part (2) of lemma \ref{lm:small_acc}, applied to $\kpl$
instead of $\kappa$, that for any $A \in [X]^{\kappa^+}$ we have $|A^\circ| = \kappa^+$.
Also, as $A^\circ$ is closed in $X$, it is $\kpl$-compact. So, to prove $A^\circ \in \mc L_\kappa$ it only remains to show
that $\Delta(A^\circ) = \kpl$.

To see this, assume that $U$ is open with $U \cap A^\circ \ne \empt$, say $x \in U \cap A^\circ$.
By the regularity of $X$ the point
$x$ has an open neighbourhood $V$ such that $\overline{V} \subs U$. Then $x \in A^\circ$
implies $|V \cap A| = \kpl$, hence $|(V \cap A)^\circ| = \kpl$. This, in turn, implies
$|U \cap A^\circ| = \kpl$ because $(V \cap A)^\circ \subs \overline{V} \cap A^\circ \subs U \cap A ^\circ$.

\medskip

(ii) Let us note first of all that if we have $A,\, B \in [X]^{\kpl}$ with  $A^\circ \setm B^\circ \ne \empt$
then there is a set $C \in [A]^{\kpl}$ such that $C^\circ \cap B^\circ = \empt$. Indeed, if
$x \in A^\circ \setm B^\circ$ and $V$ is an open neighbourhood of $x$ with $|\overline{V} \cap B| \le \kappa$ then
the choice $C = V \cap A$ clearly works.

Now we distinguish two cases.

\noindent{\bf
Case 1:}

There is a disjoint family $\{X_\xi : \xi < \kpl \} \subs [X]^{\kpl}$ such that the increasing $\kpl$-sequence
$\{Y_\xi^\circ : 0 < \xi < \kpl\}$, where $Y_\xi = \bigcup_{\eta < \xi} X_\eta$, does not
stabilize. This means that the set $I = \{\xi < \kpl : Y_{\xi+1}^\circ \setm Y_\xi^\circ \ne \empt \}$
has cardinality $\kpl$. By our above remark then for each $\xi \in I$ there is a set
$C_\xi \in [Y_{\xi+1}]^{\kpl}$ such that $C_\xi^\circ \cap Y_{\xi}^\circ = \empt$.
But this means that the members of the family $\{C_\xi^\circ : \xi \in I\}$ are pairwise
disjoint, and we are done.

\noindent{\bf
Case 2:}

For every disjoint family $\{X_\xi : \xi < \kpl \} \subs [X]^{\kpl}$ the sequence $\{Y_\xi^\circ : \xi < \kpl\}$,
as defined above, does stabilize.

Let us then fix a disjoint family $\{X_\xi : \xi < \kpl \} \subs [X]^{\kpl}$ and for any $\xi < \eta < \kpl$
put $Y_{\xi,\eta} = \bigcup \{X_i : \xi \le i < \eta \}$. But now for each fixed $\xi < \kpl$ the sequence
$$\{(Y_{\xi,\eta})^\circ : \xi < \eta < \kpl \}$$ stabilizes, i.e. there is an $\eta(\xi) < \kpl$ such that
$$(Y_{\xi,\zeta})^\circ = (Y_{\xi,\eta(\xi)})^\circ = F_\xi$$
whenever $\eta(\xi) \le \zeta < \kpl$.

The sequence $\{F_\xi : \xi < \kpl \}$ is clearly decreasing and we claim that it cannot stabilize.
Indeed, assume on the contrary that there is some $\xi_0 < \kpl$ such that $F_\xi = F_{\xi_0}$
for all $\xi_0 < \xi < \kpl$. We may then select a set $I \in [\kpl \setm \xi_0]^{\kpl}$ such that
$\eta(\xi) < \zeta$ holds whenever $\{\xi, \zeta\} \in [I]^2$ and $\xi < \zeta$.
But then the disjoint collection $\{Y_{\xi,\eta(\xi)} : \xi \in I \} \subs [X]^{\kpl}$ would witness that the set
$ F_{\xi_0}$ is \appr {{\kappa^+}} in $X$, contradicting our assumption.

Consequently, the set $J = \{\xi < \kpl : F_\xi \setm F_{\xi+1} \ne \empt\}$ has cardinality $\kpl$ and
by our introductory remark for each $\xi \in J$ there is a set $C_\xi \in [Y_{\xi,\eta(\xi)}]^{\kpl}$
such that $C_\xi^\circ \cap F_{\xi+1} = \empt$. But we also have $C_\xi^\circ \subs F_\xi$ for each $\xi \in J$,
consequently the sets $\{C_\xi^\circ : \xi \in J\}$ are pairwise disjoint, completing
the proof of (ii).

\end{proof}

\begin{proof}[Proof of theorem \ref{tm:steppingup}]
Let us assume, to begin with, that every member of $\mc L_\kappa$ is $\kappa$-resolvable.
Our aim is to show that then every member of $\mc L_\kappa$ is $\kappa^+$-resolvable.
Since $\mc L_\kappa$ is regular closed hereditary, by corollary \ref{cl:rc} it suffices
to prove that every member of $\mc L_\kappa$ has a $\kpl$-resolvable subspace.

Now, if $X \in \mc L_\kappa$ is such that its \appr {{\kappa^+}} subsets form a $\pi$-network
in $X$ then it follows from theorem \ref{tm:app} that $X$ is $\kpl$-resolvable.
Therefore, it will suffice to show that any space $X \in \mc L_\kappa$ that
has no \appr {{\kappa^+}} subset contains a $\kpl$-resolvable subspace.

To see this, we may apply lemma \ref{lm:L_k} to obtain a family
$$\{C_\alpha : \alpha < \kpl \} \subs [X]^{\kpl}$$ such that the sets $A_\alpha = C_\alpha^\circ$
are pairwise disjoint. Since each $A_\alpha \in \mc L_\kappa$ is $\kappa$-resolvable by our assumption,
it follows that the closed set
$$F = \bigcap_{\alpha < \kappa^+} \overline{\bigcup \{A_\beta  : \beta \in \kappa^+ \setm \alpha \}} \,$$
is $\kappa$-nice in the sense of definition \ref{df:nice}.

We claim that $F \in \mc L_\kappa$ holds also and this will follow if we can show that $\Delta(F) = \kpl$.
To see this, let $U$ be open with $U \cap F \ne \empt$ and pick $x \in U \cap F$.
By the regularity of $X$ the point
$x$ has an open neighbourhood $V$ such that $\overline{V} \subs U$.
Clearly, then $$I = \{\alpha : V \cap A_\alpha \ne \empt \} \in [\kappa^+]^{\kpl}.$$
Let us pick for each $\alpha \in I$ a point $x_\alpha \in V \cap A_\alpha$
and consider the set $B = \{x_\alpha : \alpha \in I\}$.
Then $B \in [X]^{\kpl}$ and $B^\circ \subs \overline{V} \cap F \subs U \cap F$,
hence $|U \cap F| = \kpl$. So we indeed have $F \in \mc L_\kappa$, and therefore $F$ is both
$\kappa$-nice and $\kappa$-resolvable.

This argument can also be applied to any regular closed subset $R \in \rcp X$
to obtain a subset of $R$ that is both $\kappa$-nice and $\kappa$-resolvable.
Thus we have concluded that the sets that are both $\kappa$-nice and $\kappa$-resolvable
form a $\pi$-network in $X$.
So if $\mc F$ is any maximal disjoint collection of such subsets of $X$
then corollary \ref{co:nice} can be applied to $X$ and $\mc F$ to conclude that $X$ has a
$\kpl$-resolvable (open) subspace, and thus the proof is completed.

\end{proof}

Since Lindelöf spaces are trivially $\omega_1$-compact, we immediately obtain from
theorem \ref{tm:steppingup} and the case $\kappa = \omega_1$ of theorem \ref{tm:gen_e_res}
the following result that was promised already in the abstract.

\begin{corollary}\label{co:Lin}
Every $\omega_1$-compact (hence every Lindelöf) regular space $X$ with
$|X| = \Delta(X) = \omega_1$ is $\omega_1$-resolvable.
\end{corollary}

We were unable to answer the natural question whether, in corollary \ref{co:Lin}, the assumption of $\omega_1$-compactness
can be relaxed to the property of having countable extent.

\end{document}